\documentclass[leqno]{amsart}
\usepackage{amssymb, amsmath, amsthm}
\usepackage[all]{xy}
\usepackage{color}
\usepackage{tikz}
\newtheorem{thm}{Theorem}[section]

\newtheorem{cor}[thm]{Corollary}
\newtheorem{prop}[thm]{Proposition}

\newtheorem{lemma}[thm]{Lemma}

\theoremstyle{definition} 
\newtheorem{defn}[thm]{Definition}
\newtheorem{eg}[thm]{Example}

\newtheorem{rmk}[thm]{Remark}

\newcommand{\calC}{\mathcal{C}}

\newcommand{\U}{\tilde{U}}

\newcommand{\into}{\hookrightarrow}

\newcommand{\orbita}{\ensuremath{\mathcal{O}_{G_1}}}
\newcommand{\orbitb}{\ensuremath{\mathcal{O}_{G_2}}}
\newcommand{\orbit}{\ensuremath{\mathcal{O}_G}}

\newcommand{\semi}[2] {{\textstyle{\int}_#2\kern-.1em#1}}
\newcommand{\F}{\underline{\Pi}_X}

\newcommand{\fg}{\Pi_G}

\newcommand{\fun}{\mathcal{F}}

\newcommand{\Cat}{\mbox{\sf Cat}}

\newcommand{\op}{{\mbox{\scriptsize op}}}

\definecolor{laura}{rgb}{.4, 0, .6}

\begin{document}
\date{\today}

\title{The Equivariant Fundamental Groupoid as an Orbifold Invariant}

\author[D. Pronk and  L. Scull]{Dorette Pronk and Laura Scull}
\address{Department of Mathematics and Statistics,
Chase Building,
Dalhousie University, 
Halifax, NS, B3H 4R2, Canada}
\address{Department of Mathematics, Fort Lewis College,  1000 Rim Drive, 
Durango, Colorado 81301-3999    
USA
}
\email{pronk@mathstat.dal.ca, scull\_l@fortlewis.edu}

\begin{abstract} We construct a 2-category version of tom Dieck's equivariant fundamental groupoid for representable orbifolds and show that the discrete fundamental groupoid is Morita invariant; hence an orbifold invariant for representable orbifolds. \end{abstract}

\thanks{The first author's research is funded by an NSERC Discovery Grant.}

\maketitle
\section{Introduction}
The equivariant fundamental category $\pi_1(G, X)$ of a $G$-space $X$ was first defined by tom Dieck in (\cite{tD}, Definition 10.7).    This category incorporates information from the fundamental groupoids of the fixed sets $X^H$ of $X$ for subgroups $H$ of $G$, combining them to create a category fibred in groupoids over the orbit category of $G$.     When $G$ is a compact Lie group, tom Dieck also defines a discrete fundamental group category $\pi_1^d(G, X)$  in (\cite{tD}, Definition 10.9), which removes some of the information coming from the topology of the group itself.     Since their introduction, these have been used   in a variety of equivariant applications, such as covering spaces, orientation theory, equivariant surgery theories,  and homology theories with twisted coefficients;  see \cite{CW, CW2, LM, LM2, MP}.  
 
  In this paper, we define a discrete tom Dieck fundamental group category for representable orbifolds.  An orbifold is representable if it is a quotient of an action of a compact Lie group on a manifold.   Many, possibly all, orbifolds can be represented  in this way.  Given such a representation, we can apply  tom Dieck's definition to create a category   $\pi_1^d(G, X)$.     However,  the $G$-space representing an orbifold  is not unique:  two translation groupoids coming from group actions represent the same orbifold if and only if they are Morita equivalent.  
  
  Morita equivalence in general corresponds to a zig-zag of essential equivalences between groupoids.  However, in  
 \cite{PS},  we showed that for translation groupoids,  Morita equivalence  is generated by two specific types of equivariant maps:  one coming from the quotient of a subgroup which acts freely, and the other by including a space into a larger space with an induced action of a larger group (Prop.  3.5 of \cite{PS}, described in detail in Section 5).  
 
 Tom Dieck's (non-discrete) $\pi_1(G, X)$ is not necessarily invariant under Morita equivalence.  However, we will build a 2-category version $\fg(X)$ of this category which is functorial with respect to equivariant maps and each of our types of Morita equivalences induces a weak equivalence of 2-categories.  So the $\fg(X)$ is an orbifold invariant.  Furthermore, when we quotient this 2-category by its 2-cells we obtain tom Dieck's discrete category $\pi_1^d(G, X)$. So we arrive at the main result of this paper:  $\pi_1^d(G, X)$ is an orbifold invariant for representable orbifolds.

\section{Background on Categorical Constructions}

In this section, we review some categorical constructions that we will use to frame our definition of tom Dieck's equivariant fundamental groupoid.  
\begin{defn} \label{groth} 
Let  $\calC$ be a 
category,  and  ${\mathcal{F}}\colon\calC^{\op}\rightarrow\Cat$ be a contravariant functor, where $\Cat$ is the 1-category of 1-categories.   
The Grothendieck  category 
$\int_{\calC}\fun$
is defined by:  
\begin{itemize}  \item An object is a pair  $(C,x)$ with $C \in \calC_0$ and  $x\in F(C)_0$.
\item  An arrow  $(g,\psi)\colon(C,x)\rightarrow(C',x')$ is a pair with 
 $g\colon C\rightarrow C'$ in $\calC_1$
and $\psi\colon x\rightarrow {\mathcal{F}}(g)(x')$ in ${\mathcal{F}}(C)_1$. 
\item Composition is defined  using the fact that  ${\mathcal{F}}(g)\colon{\mathcal{F}}(C')\rightarrow{\mathcal{F}}(C)$ is  a functor between categories:
$$
\left(\xymatrix@1{(C,x)\ar[r]^-{(g,\psi)} & (C',x')\ar[r]^-{(g',\psi')} & (C'',x'')}\right) = 
	\xymatrix@1@C=7em{(C,x)\ar[r]^-{(g'\circ g,{\mathcal{F}}(g)(\psi')\circ\psi)} & (C'',x'')},
$$
\end{itemize}
This  category  comes with a projection functor  $\int_{\calC}F\rightarrow {\calC}$.
\end{defn}

 A bicategorical version of this construction was introduced in \cite{Ba},  and full details of this construction and its properties for both the 2-categorical and bicategorical cases were worked out in \cite{Bu}.  The 2-categorical case deals with  a functor into the 3-category of 2-categories.  This is more general than what we need to express tom Dieck's constructions.    Therefore we will spell  out explicitly what Buckley's category of elements \cite{Bu} becomes in the case of  a functor $F:  C^{op} \to Cat$ when $C$ is a 2-category, and $Cat$ is the 2-category of 1-categories.    For further details we refer to \cite{Bu}.

\begin{defn}\label{groth2}
Let  $\calC$ be a 
 2-category,  and  ${\mathcal{F}}\colon\calC^{\op}\rightarrow\Cat$ be a contravariant 2-functor.  
The 2-dimensional Grothendieck construction  $\int_{\calC}{\mathcal{F}}$ is a $2$-category defined by:  
\begin{itemize} \item An object is a pair  $(C,x)$ with $C \in \calC_0$ and  $x\in F(C)_0$ as in Definition \ref{groth}. 
\item  An arrow  $(g,\psi)\colon(C,x)\rightarrow(C',x')$ is a pair with 
 $g\colon C\rightarrow C'$ in $\calC_1$
and $\psi\colon x\rightarrow {\mathcal{F}}(g)(x')$ in ${\mathcal{F}}(C)_1$ as in Definition \ref{groth}. 
\item A  2-cell
$\alpha\colon(g,\psi)\Rightarrow(g',\psi')\colon (C,x)\rightrightarrows(C',x')$
is a 2-cell $\alpha\colon g\Rightarrow g'$ in $\calC$ such that the diagram 
$$
\xymatrix{
 x\ar@{=}[d]\ar[r]^-{\psi}&{\mathcal{F}}(g)(x')\ar[d]^{{\mathcal{F}}(\alpha)_{x'}} 
\\
x \ar[r]_-{\psi'} & {\mathcal{F}}(g')(x')
}
$$
commutes in ${\mathcal{F}}(C)$.
\end{itemize}
\end{defn}
Horizontal and vertical composition of 2-cells is defined using the usual compositions of natural transformations.
It is obvious that vertical composition thus defined gives rise to the required commutative squares.
Horizontal composition is a little more involved, and we include the details here.

Let $\alpha\colon(g,\psi)\Rightarrow(g',\psi')\colon (C,x)\rightrightarrows(C',x')$ and 
$\beta\colon (h,\theta)\Rightarrow (h',\theta')\colon (C',x')\rightrightarrows(C'',x'')$ 
be 2-cells in $\int_{\calC}{\mathcal{F}}$.
Thus, we have commuting squares
\begin{equation}\label{squares}
\xymatrix{
x\ar@{=}[d]\ar[r]^-\psi & {\mathcal{F}}(g)(x')\ar[d]^{{\mathcal{F}}(\alpha)_{x'}}
\\
x\ar[r]_-{\psi'} & {\mathcal{F}}(g')(x')
}
\mbox{ in }{\mathcal{F}}(C),\mbox{ and \,\,\,\, }
\xymatrix{
x'\ar@{=}[d]\ar[r]^-\varphi & {\mathcal{F}}(h)(x'')\ar[d]^{{\mathcal{F}}(\beta)_{x''}}
\\
x'\ar[r]_-{\varphi'} & {\mathcal{F}}(h')(x'')}
\mbox{in }{\mathcal{F}}(C').
\end{equation}
Apply both ${\mathcal{F}}(g)$ and ${\mathcal{F}}(g')$ to the second square in (\ref{squares}) to obtain two commuting squares
in ${\mathcal{F}}(C)$. By the naturality of ${\mathcal{F}}(\alpha)$ the two resulting squares are part of 
a commuting cube. We append to this cube the first commuting square from (\ref{squares}):
$$
\xymatrix{
	& x\ar@{=}[dl]\ar[r]^-\psi 
	&  {\mathcal{F}}(g)(x') \ar[dl]_{{\mathcal{F}}(\alpha)_{x'}} \ar@{=}[dd]|\hole \ar[rr]^{{\mathcal{F}}(g)(\varphi)} 
	&& {\mathcal{F}}(g){\mathcal{F}}(h)(x'') \ar[dd]^{{\mathcal{F}}(g)({\mathcal{F}}(\beta)_{x''})} \ar[dl]^{{\mathcal{F}}(\alpha)_{{\mathcal{F}}(h)(x'')}}
\\
x \ar[r]_-{\psi'} & {\mathcal{F}}(g')(x')\ar@{=}[dd] \ar[rr]^(.6){{\mathcal{F}}(g')(\varphi)} 
	&& {\mathcal{F}}(g'){\mathcal{F}}(h)(x'') \ar[dd]^(.3){{\mathcal{F}}(g')({\mathcal{F}}(\beta)_{x''})}
\\
	&& {\mathcal{F}}(g)(x') \ar[dl]_{{\mathcal{F}}(\alpha)_{x'}} \ar[rr]|(.46)\hole_(.3){{\mathcal{F}}(g)(\varphi')} 
	&& {\mathcal{F}}(g){\mathcal{F}}(h')(x'')\ar[dl]^{{\mathcal{F}}(\alpha)_{{\mathcal{F}}(h')(x'')}} 
\\
	& {\mathcal{F}}(g')(x') \ar[rr]_{{\mathcal{F}}(g')(\varphi')} && {\mathcal{F}}(g'){\mathcal{F}}(h')(x'')
}
$$
Then ${\mathcal{F}}(\beta\circ\alpha)={\mathcal{F}}(g')({\mathcal{F}}(\beta)_{x''})\circ{\mathcal{F}}(\alpha)_{{\mathcal{F}}(h)(x'')}=
{\mathcal{F}}(\alpha)_{{\mathcal{F}}(h')(x'')}\circ{\mathcal{F}}(g)({\mathcal{F}}(\beta)_{x''})$ is the
diagonal of the right side of the cube, so the diagram
$$
\xymatrix@C=5em{
x\ar@{=}[d]\ar[r]^-{{\mathcal{F}}(g)(\varphi)\circ\psi} & {\mathcal{F}}(g){\mathcal{F}}(h)(x'')\ar[d]^{{\mathcal{F}}(\beta\circ\alpha)}
\\
x\ar[r]_-{{\mathcal{F}}(g')(\varphi')\circ\psi'}&{\mathcal{F}}(g'){\mathcal{F}}(h')(x'')
}
$$
commutes and defines the horizontal composition ${\mathcal{F}}(\beta\circ\alpha)$.

 \section{The Equivariant Fundamental Groupoid} \label{S:fg}

We will present tom Dieck's equivariant fundamental groupoid from \cite{tD}  from a categorical point of view, using the framework of the previous section.  

Fix a 
compact Lie group $G$.   When studying a $G$-space $X$, we often look at the diagram of its fixed sets $\{ X^H \} = \{ x \in X | hx = x \textup{ for all } h \in H \}$ for the closed subgroups $H \leq G$, with the inclusion maps between these spaces.  This is a diagram in the shape of the  orbit category $\orbit$  of $G$.  This category $\orbit$ has  objects given by  canonical orbits $G/H$ for a closed subgroup of $G$, 
with arrows defined by  $G$-equivariant maps between them.  Explicitly, an  equivariant map $ G/H \to G/K$  is defined by $H \to \alpha K$ for some $\alpha \in G$ such that $H \leq \alpha K \alpha^{-1}$;  then equivariance requires that  $gH \to g \alpha K$.  Two elements $\alpha, \beta\in G$ define the same map when $\alpha K   = \beta K$, so the map is defined by a coset $\alpha K$.  Thus  maps $G/H \to G/K$ are defined by elements $\alpha \in (G/K)^H$.  

A $G$-map $\xymatrix@1@C=1.5em{G/H\ar[r]^-x&X}$ is defined by $H \to x$ for some point $x \in X^H$;  then $gH \to gx$.  We think of such a map as defining a $G$-point in $X$, consisting of the point $x$ and its orbit $\{ gx\}$.   Then $X$ defines a contravariant functor $\Phi:  \orbit \to  {\sf Spaces}$:  $G/ H \to X^H$, and  if $\alpha:  G/H \to G/K$, then given $\xymatrix@1@C=1.5em{G/K\ar[r]^-x&X}$ we can compose and get  $\xymatrix@1@C=1.5em{G/H \ar[r]^-\alpha &G/K \ar[r]^-x & X}$.    This is defined by $H \to \alpha K \to \alpha x$ , so $x \circ \alpha $ is just given by the action $\alpha x$:  it is easy to confirm that if $\alpha \in (G/K)^H$ and $x \in X^K$, then $\alpha x \in X^H$ and the action agrees with the composition above.   In what follows, we will  use the group action notation  $\alpha x$ rather than the composition notation.

We can create equivariant invariants as diagrams indexed by $\orbit$ by applying topological invariants to the diagram of fixed sets.     Tom Dieck used this idea to create a category out of the fundamental groupoids of the fixed sets.  

\begin{defn}\label{D:td} ( \cite{tD}, Defn 10.7) Objects of $\pi_1(G, X)$ are $G$-maps $x:  G/H \to X$.  An arrow from $x:  G/H \to X$ to $y:  G/K \to X$ is a pair $(\alpha, [\gamma])$ consisting of a $G$-map $\alpha: G/H \to G/K$,  and a $G$-homotopy class (rel $G/H \times \partial I$) of $G$-paths $\gamma:  G/H \times I \to X$ such that $\gamma(0) = x$ and $\gamma(1) = \alpha y$.  

\end{defn}

We can interpret $\pi_1(G, X)$ as a  Grothendieck category as in  Definition \ref{groth}.   Define the functor   $\F$ from spaces into the category of groupoids:   $\F(G/H)  = \Pi(X^H) $ is the  fundamental groupoid of $X^H$.  Explicitly, $\F(G/H)$ is the  category  which has objects given by  equivariant maps $\xymatrix@1@C=1.5em{G/H\ar[r]^-x&X}$, or equivalently, points $x\in X^H$, and 
 arrows given by homotopy classes of paths in $X^H$. 
 
For $\alpha\colon G/H\rightarrow G/K$, we define a functor 
$\F(\alpha)\colon \F(G/K)\rightarrow\F(G/H)$. $\F(\alpha)$ is the functor $\Pi(X^K) \to \Pi(X^H)$ which sends an object $x \in X^K$
to $\alpha x \in X^H$, and a homotopy class of paths $[\gamma] $ in $X^K$ to
 $[\alpha \gamma] $ in $X^H$.    Then we define the Grothendieck category   $$ \semi \F  {{\orbit}\, }.$$   Explicitly, objects are pairs $(G/H,x)$ with $x\in X^H$,
and arrows are pairs $$(\alpha, [\gamma]) \colon (G/H, x)\to (G/K, y),$$
where $\alpha\colon G/H\rightarrow G/K$ and $[\gamma]$ is a homotopy class   (rel endpoints) of paths in $X^H$
from $x$ to $ \alpha y$, recovering  Definition \ref{D:td}.  
The resulting category is not a groupoid, but rather a category fibred in groupoids over $\orbit$. 

For non-discrete groups, tom Dieck also gives a version of the fundamental groupoid which accounts for the topology of $G$.  

\begin{defn}\label{D:tdD}  (\cite{tD}, Definition 10.9) Objects of $\pi_1^d(G, X)$ are again  $G$-maps $x:  G/H \to X$.  Morphisms are now given by equivalence classes of the morphisms of  Definition \ref{D:td},  where $(\alpha_0, [\gamma_0])$ is equivalent to $(\alpha_1, [\gamma_1])$ when there exists a $G$-homotopy $\sigma:  G/H \times I \to G/K$ between $\alpha_0$ and $\alpha_1$ and a $G$-homotopy $\Lambda:  G/H \times I \times I \to X$ such that 
\begin{align*} \Lambda (gH, 0, t) & =  x \\  \Lambda (gH, 1, t) & =  \sigma(t) y  \\  \Lambda (gH, s, 0 ) & =  \gamma_0 (s) 
\\  \Lambda (gH, s, 1 ) & =  \gamma_1 (s)  \\ \end{align*}
\end{defn}

To define $\pi_1^d(G, X)$ categorically,  we use Definition \ref{groth2},  a 2-categorical version of the Grothendieck category.  When $G$ is a compact Lie group,    $\orbit$  is a  2-category:  $\orbit[G/H,G/K]\cong(G/K)^H$ is a  topological space.  So given two arrows  $\alpha,\beta\colon G/H\rightarrow G/K$, we define a 
2-cell $[s] \colon \alpha\rightarrow \beta$  to be a homotopy  class (rel endpoints) of paths $[s]:  I\rightarrow (G/K)^H$ from $\alpha$ to $\beta$.
Then  the functor 
$\F$ defined above can be extended to 
a 2-functor $\F\colon \orbit\rightarrow\Cat$, where $\Cat$ is the 2-category of categories.    For $G/H \in \orbit$, $\F(G/H) = \Pi(X^H)$, the fundamental groupoid of $X^H$, and $F(\alpha)$ is the functor $\Pi(X^K) \to \Pi(X^H)$ as described above.  Given a $2$-cell in $\orbit$ defined by 
 a homotopy class $[\sigma]$ of paths from $\alpha$ to $\beta$ in $\orbit[G/H,G/K]$, we  define a natural transformation 
$\F([\sigma])\colon\F(\alpha)\rightarrow\F(\beta)$.    This natural transformation has components  
 $\F([\sigma])_x=\sigma x  $ for $x\colon G/K\rightarrow X$:  
  \[ \xymatrix{ \alpha  x \ar[d]_{ [\alpha\gamma]  }  \ar[rr]^{[\sigma x]} & &
{\beta x  }  \ar[d]^{ [\beta \gamma] }  \\
\alpha y \ar[rr]_{  [\sigma y] }  & &   \beta y }  \]
This is well-defined, since if  $\sigma'$ is homotopic to $\sigma$ in $(G/K)^H$, then the path defined by $\sigma' x$ is homotopic to $\sigma x$ in $X^H$.  The compositions are defined by $\sigma x * \beta \gamma$ and $ \alpha \gamma * \sigma y$, where $*$ denotes the usual concatenation of paths.  These can be written as  \begin{align*} \sigma x * \beta \gamma & = (\sigma * c_{\beta})  (c_x * \gamma)  \\   \alpha \gamma * \sigma y & = (c_{\alpha} * \sigma ) (\gamma * c_y)\end{align*}  where $c$ denotes the constant path at the respective point.  It is straightforward to show that these are homotopic by sliding $\sigma$ along from $0 \leq t \leq \frac12$ to $\frac 12 \leq t \leq 1$, and sliding $\gamma$ in the opposite way:  

 \[ \xymatrix{ \ar@{-}[d]  \ar@{-}[r]^{\sigma} \ar@{-}[dr] & \ar@{-}[r]^{c_{\beta}} \ar@{-}[dr]  &
 \ar@{-} [d]^{}  \\
\ar@{-} [r]^{c_\alpha} &\ar@{-} [r]^{\sigma} & }  \phantom{222w}  \xymatrix{ \ar@{-}[d]  \ar@{-}[r]^{c_x} &  \ar@{-}[dl] \ar@{-}[r]^{\gamma}  &\ar@{-}[dl] 
 \ar@{-} [d]^{}  \\
\ar@{-} [r]^{\gamma} &\ar@{-} [r]^{c_y} & }  \]
Therefore this defines a natural transformation as required.   


\begin{defn}  Apply  Definition \ref{groth2} to the 2-functor $\F\colon \orbit\rightarrow\Cat$.  This defines a Grothendieck 2-category:

$$\fg(X) = \semi \F  {{\orbit}\, }.$$

The category obtained by identifying all arrows connected by 2-cells will be denoted by $\fg^d(X)$.  

\end{defn}
The 2-category $\fg(X)$ and its quotient $\fg^d(X)$ will be our main objects of study from here on. We begin by explicitly describing $\fg(X)$.  Objects are pairs $(G/H,x)$ with $x\in X^H$,
and arrows are pairs $$(\alpha, [\gamma]) \colon (G/H, x)\to (G/K, y),$$
where $\alpha\colon G/H\rightarrow G/K$ and $[\gamma]$ is a homotopy  class   (rel endpoints) of paths in $X^H$
from $x$ to $ \alpha y$, as in the 1-category case.   

To describe 2-cells, let $(\alpha, [\gamma]), (\alpha', [\gamma']) \colon (G/H, x)\to (G/K, y)$ 
be maps in $\semi \F {{\orbit}\, }$.  A
$2$-cell $(\alpha, [\gamma])
\Rightarrow (\alpha', [\gamma'])$ is given by a homotopy class of paths $[\sigma]\colon  \alpha \to \alpha'$   in
$\orbit[G/H, G/K]$ such that the following diagram commutes in $\F(G/H)$:
\[ \xymatrix{ x  \ar@{=}[d]_{id_x }  \ar@{-}[rr]^{[\gamma]} & &
{\alpha y}  \ar@{-}[d]^{ [\sigma y]}  \\
x  \ar@{-}[rr]_{ [\gamma']}  & &   \alpha' y   \rlap{\quad.} }\]  
Arrows of  $\F(G/H)$ are  homotopy classes of paths in $X^H$,  and so 
this means that for any choice of representatives $\gamma$, $\gamma'$, and $\sigma$,
we have  that  the concatenation $\gamma * \sigma y$ is homotopic (rel endpoints) to $\gamma'$ in $X^H$.  
So there is a homotopy $\Lambda$ in $X^H$ that fills the following square
\begin{equation}\label{homotopy}
\xymatrix{
 x  \ar@{-}[dd]_{id_x }  \ar@{-}[rr]^{\gamma} 
\ar@{}[ddrr]|{\Lambda(\gamma,\gamma',\sigma)}& &
{\alpha y}  \ar@{-}[dd]^{\sigma y}  \\
 &  \\
x  \ar@{-}[rr]_{ \gamma'}  & &   \alpha'  y\rlap{\quad.}
}
\end{equation}

Then $\pi^d_1(G, X)$ from Definition \ref{D:tdD} is exactly the  quotient  of the 2-category  $\fg(X)$ by these $2$-cells.  So tom Dieck's discrete category is exactly $\Pi_G^d(X)$.  

For future use, we explicitly describe the horizontal composition of 2-cells for this category.  
Given a 2-cell $[\sigma]$ from $(\alpha, [\gamma])$ to $(\alpha', [\gamma'])$ and a 2-cell $[\omega]$ from $(\beta, [\zeta])$ to $(\beta', [\zeta'])$, we know that we have homotopies   \[ \xymatrix{ x \ar@{=}[d]  \ar[rr]^{\gamma}  & &
 \alpha y   \ar[d]^{ \sigma y}  \\
x  \ar[rr]_{\gamma'}  & &  \alpha' y  }  \phantom{ww} \textup{and} \phantom{wwww}
\xymatrix{ y \ar@{=}[d]  \ar[rr]^{\zeta}   & &
 \beta z  \ar[d]^{ \omega z}  \\
y  \ar[rr]_{\zeta'}  & &  \beta' z  \rlap{\quad.} }   \]    The composition of $[\sigma]$ and $[\omega]$  is defined by applying both $\alpha $ and $\alpha'$ to the second square, to get 
    \[ \xymatrix{ \alpha y \ar@{=}[d]  \ar[rr]^{ \alpha \zeta}  & &
 \alpha \beta z  \ar[d]^{ \alpha \omega z}  \\
\alpha y   \ar[rr]_{\alpha \zeta'}  & &  \alpha' y   }  \phantom{ww} \textup{and} \phantom{wwww}
\xymatrix{ \alpha' y \ar@{=}[d]  \ar[rr]^{\alpha' \zeta}   & &
\alpha' \beta z  \ar[d]^{ \alpha' \omega z}  \\
\alpha' y  \ar[rr]_{\alpha' \zeta'}  & &  \alpha' \beta' z  \rlap{\quad.} }   \]    

Then we form the diagram 
$$
\xymatrix{
	& x\ar@{=}[dl]\ar[r]^{\gamma}
	&  { \alpha \zeta} \ar[dl]_{\sigma y } \ar@{=}[dd]|\hole \ar[rr]^{\alpha \zeta} 
	&& \alpha \beta z \ar[dd]^{\alpha \omega z} \ar[dl]^{\sigma \beta z}
\\
x \ar[r]_{\gamma'} & \alpha' y \ar@{=}[dd] \ar[rr]^(.6){\alpha' \zeta} 
	&&\alpha' \beta z \ar[dd]^(.3){\alpha' \omega z}
\\
	&&\alpha y  \ar[dl]_{\sigma y} \ar[rr]|(.46)\hole_(.3){\alpha \zeta'} 
	&& \alpha \beta' z \ar[dl]^{\sigma \beta' z }
\\
	& \alpha' y  \ar[rr]_{\alpha ' \zeta' } && \alpha' \beta' z
}
$$
Define the composition by  $$ [\sigma][ \omega] = [S]= [\sigma \beta * \alpha' \omega] \simeq [\alpha \omega * \sigma \beta']. $$  This is a 2-cell from $(\alpha \beta, [\gamma * \alpha \zeta])$ to $(\alpha'\beta', [\gamma' * \alpha' \zeta'])$ as required, since we have a homotopy
 \[ \xymatrix{ x \ar@{=}[d]  \ar[rr]^{\gamma* \alpha \zeta}  & &
 \alpha \beta  z   \ar[d]^{ S z}  \\
x  \ar[rr]_{\gamma' * \alpha' \zeta'}  & &  \alpha'  \beta' z  \rlap{\quad.} } \]

\begin{eg} \label{E:s1}
Let $G = \mathbb{Z}/2 = \{ e, \tau\}$  and $X = S^1$ where $\tau$ acts on $X$ by reflection through a horizontal axis.  Then $X^G = \{E, W\}$, where $E$ and $W$ are the east and west cardinal points on the circle, and $X^e = S^1$.  Then the objects of $\Pi_G(X)$ consist of $(G/e, \lambda)$ for any $\lambda \in S^1$, and the points $(G/G, E)$ and $(G/G, W)$.      

In between the points  $(G/e, \lambda_1)$ and $(G/e, \lambda_2)$, there are arrows of type $(e, [\gamma])$ where $\gamma$ is a path in $S^1$ from $\lambda_1$ to $\lambda_2$; these correspond to winding numbers in  $\mathbb{Z}$.  Additionally, there are arrows of type $(\tau, [\zeta])$ where $\zeta$ is a path from $\lambda_1$ to $\tau \lambda_2$.    Together with the arrows of the first type, we get a $D_{\infty}$ worth of arrows between points of this type.   

There are also  arrows from $(G/e, \lambda)$ to $(G/G, E)$, again corresponding to winding numbers in $\mathbb{Z}$; and similarly to $(G/G, W)$.  
Since $G$ is discrete, there are no non-trivial 2-cells; so $\Pi^d_G(X) = \Pi_G(X) $ in this case.   A skeleton of this category is given below;  for more details, see \cite{BS}. 
\[ \xymatrix{ & (G/e, \lambda) \ar[dr]^{\mathbb{Z}} \ar[dl]_{\mathbb{Z}}\ar@(ul, ur)[]^{D_{\infty} }\\
(G/G,E)  & & (G/G,W) } \]
\end{eg}

\begin{eg} \label{E:t2}
Let $G = S^1$, the circle group, and $X= T^2 = S^1 \times S^1$, where the action is given by $\lambda(\alpha, \beta) = (\alpha, \lambda \beta)$.    Then the fixed set $X^H$ is empty for any $H \neq \{ e\}$, so the objects of $\Pi_G(X)$ are all of the form $(G/e, (\alpha, \beta))$.    Between objects $(G/e, (\alpha_1, \beta_1))$ and $(G/e, (\alpha_2, \beta_2))$, there are maps defined by $(\lambda, [\gamma])$  for $\lambda\in S^1$, where $\gamma$ is a path in $T^2$ from $(\alpha_1, \beta_1) $ to $(\alpha_2, \lambda \beta_2)$.   Again, we can classify $[\gamma]$ using winding numbers in the $\alpha$ and $\beta$ coordinates, so there are $S^1 \times ({\mathbb Z}\oplus {\mathbb Z}) $ worth of arrows  between any two points.  This time, however, $G$ has non-trivial topology, so many of these are connected by 2-cells.  In fact, we can look at the projection of $\gamma$ onto the $\beta$ coordinate and use this to define a path $\sigma$ in $S^1$.  Then we have a  2-cell from $(e, [\gamma])$ to $(e, [\gamma * \sigma^{-1} (\alpha_2, \beta_2)])$, showing that every arrow has a 2-cell connecting it to an arrow which has the same projection onto the first coordinate, but is constant in the second coordinate:

\[ \xymatrix{
 (\alpha_1, \beta_1)  \ar@{-}[d]_{id }  \ar@{-}[rr]^{\gamma} 
& &
{ (\alpha_2, \lambda \beta_2) }  \ar@{-}[d]^{\sigma^{-1} (\alpha_2, \beta_2)} \\
  (\alpha_1, \beta_1)  \ar@{-}[rr]_{ \gamma'}  & &  (\alpha_2, \beta_1) \rlap{\quad.}
}\]

\[\scalebox{.15}{\includegraphics{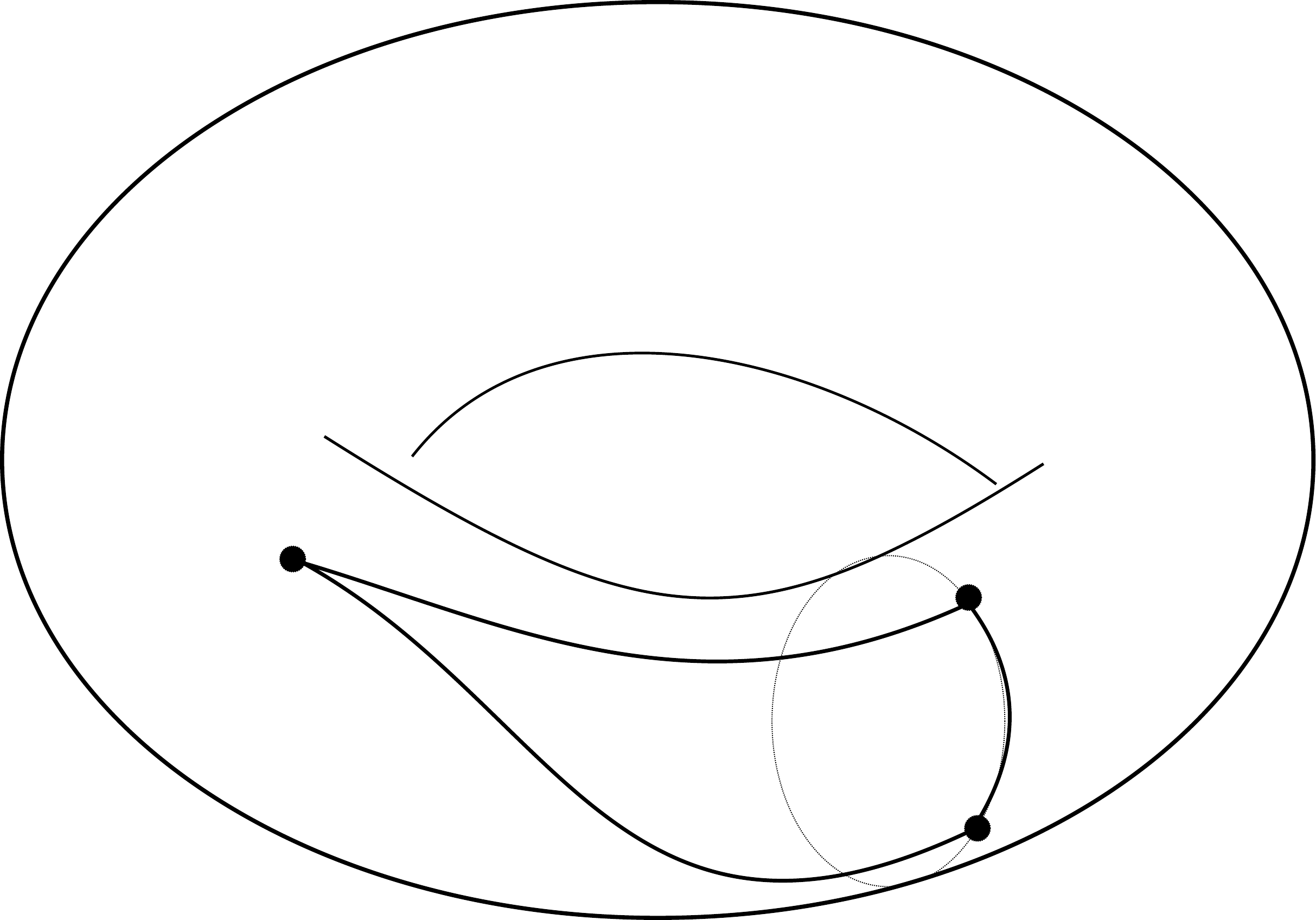}}\]
Thus two arrows are connected by 2-cells if the projections of the paths onto the first coordinates are homotopic, and $\Pi^d_G(X)$ has a  $\mathbb Z$ worth of paths between any two points, a category equivalent to the fundamental groupoid of the circle.  
\end{eg}

\begin{rmk}
Tom Dieck also defined a notion of an equivariant $\pi_0$, such that each $\pi_0(G,X)$ can be viewed as a Grothendieck category but this notion is not invariant under Morita equivalence.
\end{rmk}

\section{Functoriality of $\Pi_G(X)$}

In this section, we study the functoriality of the $\fg(X)$ construction.  We start with  the category of compact Lie groups acting smoothly on manifolds, with smooth equivariant maps between them.  In preparation for what follows,  we will consider a   $G$-space $X$ as represented by a smooth translation groupoid $G \ltimes X$.  
\begin{defn} \label{D:trans}
Let $G$ be a Lie group with a smooth  left action on a manifold $M$.
Then the translation groupoid (also called action groupoid)  $G\ltimes M$ is defined as follows.   
The space of objects is  $M$, and the space of arrows is defined to be   $G\times M$ with   the 
source of an arrow $(g, x)$ defined by  $s(g, x) = x$, and the target by using the action of $G$ 
on $M$, $t(g, x) = gx$.    So $(g, x)$ is an arrow from  $x $ to $ gx$.  
The units map is defined by  $u(x)=(e,x)$, where $e$ is the identity element in $G$, composition by 
 $(g',gx)\circ(g,x)=(g'g,x)$, and inverses by $(x, g)^{-1} = (gx, g^{-1})$.  
\end{defn}

\begin{defn} \label{D:transmap}
An {\em equivariant map} $G\ltimes X\rightarrow H\ltimes Y$ between translation groupoids is a pair $(\varphi,f)$, 
where $\varphi\colon G\rightarrow H$
is a group homomorphism and $f\colon X\rightarrow Y$ is a $\varphi$-equivariant smooth map, 
{\it i.e.}, $f(gx)=\varphi(g) f(x)$ 
for $g\in G$ and $x\in X$.
\end{defn}

Definition \ref{D:transmap}  gives a smooth functor between the topological groupoids of Definition \ref{D:trans}.  
We want to consider the 2-category of translation groupoids, with 2-cells given by smooth natural transformations.  
Suppose that $r$ is a natural transformation of smooth groupoids,
$r\colon(\varphi_1,f_1)\Rightarrow(\varphi_2,f_2)\colon G_1\ltimes X_1\rightrightarrows G_2\ltimes X_2$.  Then $r$ is  defined by a smooth map $r\colon X_1\rightarrow G_2 \times X_2$ such that $r(x):  f_1(x) \to f_2(x)$ for all $x\in X_1$,  satisfying the naturality condition.  So  $r$ is determined by a map $r:  X_1 \to G_2$ such that $r(x) = r_x$ satisfies $r_xf_1(x) = f   _2(x)$.   The naturality condition requires that the diagram commutes in $G_2 \ltimes X_2$:
 \[ \xymatrix{ f_1(x)  \ar[d]_{\varphi(g)} \ar[rr]^{r_x} & &
f_2(x)   \ar[d]^{\varphi_2(g)}  \\
f_1(gx) = \varphi_1(g)f_1(x)   \ar[rr]_{ r_{gx}}  & &  f_2(gx) = \varphi_2(g) f_2(x) \rlap{\quad.} }\]   Thus, $r_{gx}\varphi_1(g) = \varphi_2(g) r_x$ for  all $g \in G_1$ and $x \in X_1$.

We will denote the 2-category of smooth translation groupoids,  equivariant maps and natural transformations 
by {\sf EqTrGpd}.    Our goal in this section is to show that $\Pi_G(X)$ is functorial with respect to this category.  
We start by considering the effect of a map as in Definition \ref{D:transmap} on the orbit category.  

\begin{lemma}\label{L:orbit}  If $\varphi:  G_1 \to G_2$ is a group homomorphism, then we have a functor $\orbita \to \orbitb$ defined by $\varphi(G_1/H) = G_2/\varphi(H)$, and $\varphi(\alpha)$ for $\alpha:  G_1/H \to G_1/K$.  

\end{lemma}

\begin{proof}  The map $\alpha$ is defined by a coset $\alpha K$ where $H \subseteq \alpha K \alpha^{-1}$.  Then $\varphi(\alpha) \varphi(K)$ defines a coset and $\varphi(H) \subseteq \varphi(\alpha)\varphi( K ) \varphi(\alpha)^{-1}$.

\end{proof}
Next, we verify that an equivariant map respects fixed sets.  

\begin{lemma}  \label{L:fix}
 Let $(\varphi,f)\colon G_1\ltimes X_1\rightarrow G_2\ltimes X_2$ be a homomorphism of translation groupoids.  If  $x \in X_1^H$ for $H \leq G_1$,  then $f(x) \in X_2^{\varphi(H)}$. 
\end{lemma}

\begin{proof}   Suppose that $x \in X_1^H$, so $hx = x$ for all $h \in H$.  Then $f(hx) = f(x)$, and so $\varphi(h) f(x) = f(x)$.  So $f(x) \in X_2^{\varphi(H)}$ as claimed.  

\end{proof}

Now we show that $\fg(X)$ is functorial with respect to the maps of Definition \ref{D:transmap}.

\begin{prop}\label{Pi-functor}   Let $(\varphi,f)\colon G_1\ltimes X_1\rightarrow G_2\ltimes X_2$.  Then we get an induced functor between the 2-categories $\Pi(\varphi,f):     \Pi_{G_1}(X_1)\rightarrow\Pi_{G_2}(X_2)$   \end{prop}

\begin{proof}    We have shown in Lemma \ref{L:orbit} that $G_1/H \to G_2/\varphi(H)$ defines a functor from \orbita to \orbitb.     We define a functor $F = \Pi(\varphi,f):   \Pi_{G_1}(X_1)\rightarrow\Pi_{G_2}(X_2)$  as follows:

\noindent{\bf Objects:}
$F(G_1/H,x) = (G_2/\varphi(H),f(x))$.     If  $x\in X_1^H$, then Lemma \ref{L:fix} shows that $f(x)\in X_2^{\varphi(H)}$, so this is well-defined.  

\noindent {\bf Arrows:}
 If $(\alpha,[\gamma])\colon (G_1/H,x)\rightarrow(G_1/K,y)$ is an arrow in $\Pi_{G_1}(X_1)$, then $\alpha:  G_1/H \to G_1/K$ and $[\gamma]$ is a homotopy class of paths $x \to \alpha y$ in $X^H$.  Then $f[\gamma] = [f\gamma]$  defines a homotopy class of paths from $f(x)$ to $f(\alpha y)  =  \varphi(\alpha) f(y)$, defining an arrow $F(\alpha,[\gamma]):  
(G_2/\varphi(H),f(x))\rightarrow(G_2/\varphi(K),f(y))$ in $\Pi_{G_2}(X_2)$.

\noindent{\bf 2-cells:}   Suppose that $[\sigma]$ is a homotopy class of maps  $G_1 /H \to G_1/K$, given by    $ \sigma:  I \to (G_1/H)^K$ from $\alpha$ to $\beta$ such that   the following commutes:  \[ \xymatrix{ x  \ar@{=}[d]_{id_x }  \ar@{-}[rr]^{[\gamma]} & &
{\alpha y}  \ar@{-}[d]^{ [\sigma y]}  \\
x  \ar@{-}[rr]_{ [\gamma']}  & &   \alpha' y   \rlap{\quad.} }\]  So we have a homotopy $\Lambda$ from $\gamma * \sigma y$ to $\gamma'$.  
 Define $F([\sigma)]=[\varphi(\sigma)]:  \varphi(\alpha) \to \varphi(\beta)$.  This gives a homotopy class of paths in  $(G_2/\varphi(H))^{\varphi(K)}$, and $f(\Lambda)$  provides a homotopy from $f(\gamma) * \varphi(\sigma) f(y)$ to $f(\gamma')$  making the following diagram commute:  
 \[ \xymatrix{ f(x)  \ar@{=}[d]_{id_{f(x)} }  \ar@{-}[rr]^{[\gamma]} & &
{\varphi(\alpha ) f(y)}  \ar@{-}[d]^{ [\varphi(\sigma) f(y)]}  \\
f(x)  \ar@{-}[rr]_{ [\gamma']}  & &  \varphi( \alpha') f(y)   \rlap{\quad.} }\]  
\end{proof}

Lastly, we consider the effect of the 2-cells of {\sf EqTrGpd}.  
\begin{prop} \label{P:pnt} Let $r$ be a natural transformation of smooth groupoids,
$r\colon(\varphi_1,f_1)\Rightarrow(\varphi_2,f_2)\colon G_1\ltimes X_1\rightrightarrows G_2\ltimes X_2$.  Then we have an induced pseudo natural transformation $\Pi(r)\colon \Pi(\varphi_1,f_1)\Rightarrow\Pi(\varphi_2,f_2)$.
\end{prop} 

\begin{proof}

Given a natural transformation  $r$, we will define  a map from $\Pi_{G_1}(X_1)_0 \to  \Pi_{G_2}(X_2)_1 $  such that an object  $(G/H, x)$ corresponds to an arrow  $(\beta, [\gamma]) $ where $\gamma$ is a path from $ f_1(x)  $ to $ \beta f_2(x)$ in $X^{\varphi(H)}_2$.      
First   we claim that $\varphi_1(H) = r_x^{-1}\varphi_2(H) r_x$:  we know that  $r_{hx} \varphi_1(h) = \varphi_2(h) r_x$.  But since $x \in X^H$, $hx =x$ and so $r_{x} \varphi_1(h) = \varphi_2(h) r_x$, and hence $\varphi_1(h) = r_x^{-1} \varphi(h) r_x$.   Therefore $r^{-1}_x$ defines a map $G_2/\varphi_1(H) \to G_2/\varphi_2(H)$ in $\orbitb$. 

So now define the components of the pseudo natural transformation:   to every point $(G/H, x) $  we assign the arrow  $$(r_x^{-1}, [c_{f_1(x)}]) :  (G_2/\varphi_1(H), f_1(x)) \to (G_2/\varphi_2(H), f_2(x))$$ where $c_{f_1(x)}$ is the constant path from $f_1(x)$ to $r_x^{-1}f_2(x) = f_1(x)$.  Now let  $(\alpha, [\gamma]):  (G_1/H, x) \to (G_1/K, y)  $ be an arrow of $\Pi_{G_1}(X_1)$, and consider the naturality square  \[ \xymatrix{ f_1(x)  \ar[d]_{(\varphi(\alpha), [f_1(\gamma)])} \ar[rr]^{(r_x^{-1}, [c_{f_1(x)}])} & &
f_2(x)   \ar[d]^{(\varphi_2(\alpha), [f_2(\gamma)])}  \\
f_1(gx)    \ar[rr]_{ (r_y^{-1}, [c_{f_1(y)}]) }  & &  f_2(y)   \rlap{\quad.} }\]     Comparing the two compositions, we see that going one way we have $$(r_x^{-1}\varphi_2(\alpha), [c_{f_1(x)}* r_x^{-1}  f_2(\gamma) ] ) =(r_x^{-1}\varphi_2(\alpha), [ r_x^{-1}f_2(\gamma)])$$   Going the other way gives  $$(\varphi_1(\alpha) r_y^{-1}, [f_1(\gamma) *c_{f_1(y)}] ) =(\varphi_1(\alpha) r_y^{-1}, [f_1(\gamma)] ).$$ These are not equal, so this is not a strict natural transformation.  

To create the pseudo natural transformation,  to every morphism $(\alpha, [\gamma])$ we need to assign a 2-cell $[s]:  \varphi_1(\alpha) r_y^{-1} \to r_x^{-1} \varphi_2(\alpha)$ which will fill in the diagram
 \[ \xymatrix{ f_1(x) \ar@{=}[d]  \ar[rr]^{f_1(\gamma)}  & &
\varphi_1(\alpha) r_y^{-1} f_2(y)   \\
f_1(x)  \ar[rr]_{r^{-1}_x f_2(\gamma)\phantom{ww}}  & &  r_x^{-1} \varphi_2(\alpha) f_2(y)  }  \] and satisfy the required coherence properties.   We will leverage the fact that we know that $f_1(\gamma) = r^{-1}_{\gamma} f_2 (\gamma)$.  So we are trying to compare the following:  
 \[ \xymatrix{ f_1(x) \ar@{=}[d]  \ar[rr]^{r^{-1}_{\gamma} f_2(\gamma) \phantom{wwwwW}}  & &
   r^{-1}_y \varphi_2(\alpha) f_2(y)   \\
f_1(x)  \ar[rr]_{r^{-1}_x f_2(\gamma)\phantom{ww}}  & &  r_x^{-1} \varphi_2(\alpha) f_2(y)  }  \] 
 where across the top  the adjustment elements $r_{\gamma}$ are changing while  the single adjustment element $r_x^{-1}$ has been used along the bottom.  Written in this form, it becomes evident that we can define a 2-cell  $$s(\alpha, [\gamma])(t) = s(t)  = r^{-1}_{\gamma(t-1)} \varphi_2(\alpha).$$     We can define a homotopy between $f_1(\gamma) * s(t) f_2(y) $ and $r_x^{-1}f_2(\gamma) $ by writing the first map as \begin{align*} f_1(\gamma) * s(t) f_2(y) & = \left( r^{-1}_{\gamma(t) } f_2(\gamma) \right) * \left( r^{-1} _{\gamma(1-t) }\varphi_2(\alpha) f_2(y)  \right) \\ & = \left( r^{-1}_{\gamma(t)} * r^{-1} _{\gamma(1-t)} \right) \left( f_2(\gamma) * c_{f_2(\alpha y)} \right) \end{align*} Then it is easy to write a homotopy collapsing the first part to the constant map $r^{-1}_x$ and expanding the second  to $f_2(\gamma)$:

\[ \xymatrix{ \ar@{-}[d]  \ar@{-}[r]^{r^{-1}_{\gamma(t)} } \ar@{-}[dr] & \ar@{-}[r]^{r^{-1}_{\gamma(1-t)}}  &\ar@{-}[dl] 
 \ar@{-} [d]^{}  \\
\ar@{-} [r]_{c_x} &\ar@{-} [r]_{c_x} & }  \phantom{222w}  \xymatrix{ \ar@{-}[d]  \ar@{-}[r]^{f_1(\gamma)} &  \ar@{-}[dr] \ar@{-}[r]^{c_{f_2(\alpha y)} }  &
 \ar@{-} [d]^{}  \\
\ar@{-} [rr]_{f_1(\gamma)} &  & }  \]     
Coherence conditions can be verified by a straightforward calculation.

\end{proof}

\section{The Fundamental Category as an Orbifold Invariant}   \label{S:trans}

The classical definition of orbifolds  is a  generalization of the definition of 
manifolds based on charts and atlases, where the local neighbourhoods 
are homeomorphic to $U = \U/G$ where $G$ is a finite group acting on an open set 
$\U \subseteq {\mathbb R}^n$. Note that we will not  require that $G$ acts effectively on $\U$.
An orbifold can then be defined in terms of an 
 orbifold atlas, which is a locally compatible family of charts $(\U, G)$ such that the sets 
$\U/G$ give a cover of $M$.    The usual notion of equivalence of atlases through 
common refinement is used; details can be found in \cite{sa56,sa57} for the effective case and in \cite{PST} for the more general case.

Working with orbifold atlases is cumbersome, and an alternate way of representing orbifolds using groupoids 
has been developed.  It was shown in \cite{MP1} that every smooth orbifold can be represented 
by a Lie groupoid, which is determined up to essential equivalence.   
We are interested in representing orbifolds by a particular type of  Lie groupoid:  translation  groupoids $G \ltimes M$ coming from the smooth  action of a compact Lie group $G$ 
on a manifold $M$ as in Definition \ref{D:trans}, where all of the isotropy groups are finite.   Many, possibly  all,  orbifolds can be represented this way.
Satake showed  that every effective orbifold can be represented in this way \cite{sa57}.
His proof does not hold for non-effective orbifolds, but a partial result was obtained by  Henriques and Metzler \cite{HM}; 
their Corollary 5.6 shows that  all 
orbifolds for which all the ineffective isotropy groups have trivial centers are representable.   
It is conjectured that all orbifolds are representable, but this is currently unknown.  

For  this paper, we restrict our attention to those orbifolds 
that are representable, so that we can work with their translation groupoids.    In what follows, we think 
 of  a groupoid of  {\sf EqTrGpd}  as a representation of its underlying quotient space, 
encoding this orbifold and its  singularity types.  However, this representation is not unique;  
the same orbifold can be represented by different groupoids.  To represent orbifolds,  the  equivalence relation on groupoids is generated by the notion of essential equivalence 
for categories internal to the category of smooth manifolds:   an equivalence of categories that respects the topology on the groupoids.  It is shown in \cite{PS} that for representable orbifolds, we only need to look at equivariant essential equivalences, that is, essential equivalences of the form of Definition \ref{D:transmap}.   Thus we can describe the category of representable orbifolds as a bicategory of fractions  $\mbox{\sf EqTrGpd}(W^{-1})$, where 
the objects of this category are  Lie translation groupoids with finite isotropy groups, 
and a  morphism ${\mathcal G}\rightarrow{\mathcal H}$ is a span of homomorphisms 
$$\xymatrix{{\mathcal G}&{\mathcal K}\ar[l]_\omega \ar[r]^\varphi & {\mathcal H}},$$
where $\omega$ is an equivariant essential equivalence.   Such morphisms are also called {\em generalized maps}.  
We can think of this as replacing  the source groupoid $\mathcal{G}$ with an essentially equivalent groupoid  $\mathcal{K}$ and then using the new representative $\mathcal{K}$ to define our map.      

The main results of this paper are:  

\begin{thm} \label{T:maineq}
For every equivariant essential equivalence  of smooth translation groupoids
$$
\varphi\colon G\ltimes X\to H\ltimes Y
$$
the induced functor between the equivariant fundamental groupoids,
$$
\Pi(\varphi)\colon \Pi_G(X)\to\Pi_H(Y)
$$
is a weak equivalence of 2-categories.
\end{thm}

\begin{cor}
For every equivariant essential equivalence  of smooth translation groupoids
$$
\varphi\colon G\ltimes X\to H\ltimes Y
$$
the induced functor between the discrete equivariant fundamental groupoids,
$$
\Pi^d(\varphi)\colon \Pi^d_G(X)\to\Pi^d_H(Y)
$$
is a weak equivalence of categories.
\end{cor}

There are a couple of obvious types of maps which are equivariant  essential equivalences:  
if we have a $G$-space $X$ such that a normal subgroup $N$ of $G$ acts freely on $X$, 
then  the quotient map 
\begin{equation}\label{quotientform}
G\ltimes X \rightarrow G/N\ltimes X/N,
\end{equation}
is an essential equivalence.    Similarly, for any (not necessarily normal) subgroup $L$ of a group 
$G$ and $L$-space $X$, we can create a  $G$-space 
$G \times _L X = G \times Z / \sim$, where $[g\ell, z] \sim [g, \ell z]$ for any $\ell \in L$, and 
 the inclusion $X \to G \times_L X$ defined by $x \to [e, x]$ gives an essential equivalence 
\begin{equation}\label{inclform}
L\ltimes X\rightarrow G\ltimes (G\times_L X).
\end{equation}

It turns out that these are the only forms of equivariant weak equivalences we need to deal with,
since they generate all other equivariant essential equivalences through composition.

\begin{prop}[Prop.  3.5 \cite{PS}]  
Any equivariant essential equivalence is a composite of maps of the forms (\ref{quotientform}) 
and (\ref{inclform}) described above.
\end{prop}

Thus we have an explicit description of the equivalences in  $\mbox{\sf EqTrGpd}(W^{-1})$, and can verify that we have an orbifold invariant by checking its behaviour  under these two types of maps.  The proof of Theorem \ref{T:maineq} will follow from Propositions \ref{P:free} and \ref{P:incl} below, showing that $\Pi$ induces a weak equivalence of 2-categories under each of these types of equivariant essential equivalences. 

We will start by considering the case where a normal subgroup $N$ of $G$ acts freely.  

\begin{defn}\label{D:barnot}
We set the following notation:  $X$ is a $G$ space and $N$ is a normal subgroup acting freely.  Then we denote the projection  by $p:G \to G/N$, and abbreviate $G/N = \bar{G}$, $p(g) = \bar{g}$ and $p(H) = \bar{H}$ for any subgroup $H \subseteq G$.    Similarly, we denote the  projection $p:  X \to X/N$, and $p(x) = \bar{x}$.  
\end{defn}

Proving the equivalence of fundamental categories in this case will use some lifting of paths, and we will make use of the following lemma comparing different lifts.

\begin{lemma} \label{L:lift}  Suppose that $G$ is a compact Lie group acting smoothly on  $X$ and $N$ is a normal subgroup of $G$ which acts freely on $X$.  Suppose that $\gamma$ and $\zeta$ are two paths $[0,1] \to X^H$ such that $\gamma(0)=\zeta(0)$ and $p(\gamma) = p(\zeta)$ in $\bar{X}^{\bar{H}}$.  Then there exists a continuous path $\eta:  I \to G$ with $\eta(0) = e$, such that $\gamma(t) = \eta(t) \zeta(t)$ and $\eta(t)  \in N(H)$ for all $t \in [0, 1]$.    
    
\end{lemma}

\begin{proof}
 Consider the map $\epsilon:  N \times X \to X \times_{\bar{X}} X$ defined by $\epsilon(n, x) = (nx, x)$, where $X \times_{\bar{X}}X = \{ (x, y) | \bar{x} = \bar{y} \}$. 
 Because $N$ acts freely on $X$, this map is  a homeomorphism.  Now consider the map $I \to (X \times_{\bar{X}}X)$ defined by $(\zeta, \gamma)$.  Composing with $\epsilon^{-1}$ gives a map  $I \to N \times X$ defined by $(n_t, x_t) = (\eta, \gamma)$ where $\eta(t) \gamma(t)  = \zeta(t)$.  Since $\gamma(0) = \zeta(0)$ we must have $\eta(0) = e$.
 
 Now we know that $\gamma$ and $\zeta$ lie in $X^H$.  Thus if $h \in H$, then $h \gamma(t) = \gamma(t)$ and $h \zeta(t) = \zeta(t)$, so  $h\eta(t)  \gamma(t) = \eta(t) \gamma (t)$.   So $\eta(t) h\gamma(t)  = \eta(t) \gamma(t) = h\eta(t) \gamma(t)  = \eta'(t)h\gamma(t)$, where $\eta'(t)  = h\eta(t) h^{-1}$.  This means that $\eta'(t)  h\gamma(t) = \eta(t) h \gamma(t)$, and since $N$ acts freely,  $\eta(t) = \eta'(t)$.  Thus,  $h \eta(t) h^{-1} = \eta(t)$, so $\eta(t) \in N(H)$.  
\end{proof}

\begin{prop}   \label{P:free}  Suppose that $G$ is a compact Lie group acting smoothly on  $X$, and $N$ is a normal subgroup of $G$ which acts freely on $X$.  Then  the map $\Pi_G(X) \to \Pi_{G/N}(X/N)$ induced by the projection  is an equivalence of $2$-categories.  
\end{prop}

\begin{proof} Throughout this proof, we use the notation from Definition \ref{D:barnot}.

{\bf Surjective on Objects:}    Suppose that $(\bar{G}/\bar{L}, \bar{x}) $ is an object of $\Pi_{\bar{G}} (\bar{X})$.  Let $x\in X$ such that $p(x) = \bar{x}$, and define $L = p^{-1}(\bar{L}) \cap I_x$ where $I_x$ denotes the isotropy of $x$.  Then $L \subseteq I_x$ and so $x \in X^L$.  We claim that $p(L) = \bar{L}$:  if $\bar{\ell} \in p(L)$, then $\bar{\ell} = p(\ell)$ for $\ell \in L$, and so $\ell \in p^{-1}(L)$ and hence $\bar{\ell} \in \bar{L}$.  Conversely,  if $\bar{\ell} \in \bar{L}$, then choose $\ell$ in $G$ such that $p(\ell) = \bar{\ell}$.  So since $\bar{\ell} \bar{x} = \bar{x}$, we know there exists $n \in N$ such that $n\ell x = x$.  So $n \ell \in I_x$ and $n \ell \in p^{-1}(\bar{L})$, and so $n \ell \in L$ such that $p(n \ell) = \bar{\ell}$, showing that $\bar{L} \subseteq p(L)$.   
Therefore $(G/L, x)$ is an element of $\Pi_G(X)$ such that $p(G/L, x) =( \bar{G}/\bar{L}, \bar{x})$. 

{\bf Surjective on Arrows:} Suppose we have objects  $(G/H, x)$ and $(G/K, y)$ of $\Pi_G(X)$,  and an arrow $(\bar{G}/\bar{H}, \bar{x}) \to (\bar{G}/\bar{K}, \bar{y})$ in $\Pi_{\bar{G}}(\bar{X})$.  This arrow is given by $(\bar{\alpha}, [\bar{\gamma}])$ where $\bar{\alpha}:  \bar{G}/\bar{H} \to \bar{G}/\bar{K}$, so $\bar{\alpha} \in \bar{G}/\bar{K}$ with $\bar{\alpha}^{-1} \bar{H} \bar{\alpha} \subseteq \bar{K}$, and $\bar{\gamma}$ represents a homotopy class of paths $\bar{\gamma}:  \bar{x} \to \bar{\alpha} \bar{y}$ in $\bar{X}^{\bar{H}}$.

First, choose $\hat{\alpha} \in G $ such that $p(\hat{\alpha}) = \bar{\alpha}$.  
Next,  $p:  X^H \to \bar{X}^{\bar{H}}$ is the quotient of the free action of $N$, so this is a principal fibration and we can lift $\bar{\gamma}$ to $\gamma \in X^H$ starting at $x$.  So we have $\gamma:  x \to z$ in $X^H$, and we know that $\bar{z} = \bar{\alpha} \bar{y}=  \overline{\hat{\alpha} y}$.  Therefore there exists $n \in N$ such that $n \hat{\alpha} y = z$.    Define $\alpha = n\hat{\alpha}$, so that $\gamma:  x \to \alpha y$ in $X^H$.   Then $p(\alpha, [\gamma]) = (\bar{\alpha}, [\bar{\gamma}]) $.   So we just need to show that $\alpha:  G/H \to G/K$.  To see this, we need to show that $\alpha H \alpha^{-1} \subseteq K$.

Let $h \in H$.  Then $\bar{h} \in \bar{\alpha} \bar{K} \bar{\alpha} ^{-1} $, and  $h = \alpha k \alpha^{-1} n$ for some $n \in N$.  Now $ky = y$ for any $k \in K$, and we know that $\alpha y \in X^H$, so $h\alpha y = \alpha y$ for any $h \in H$, and hence $\alpha^{-1}h \alpha y = y$.     Thus $y = ky = \alpha^{-1} h \alpha y = \alpha^{-1} (\alpha k \alpha^{-1} n) \alpha y = k \hat{n} y$.   So $y = \hat{n}y$ and $\hat{n} = e$ since $N$ acts freely.  So $h = \alpha k \alpha^{-1}$ as required.  

{\bf Full on 2-cells:}   Suppose we have two arrows $(\alpha, [\gamma])$ and $(\beta, [\zeta])$ from $(G/H, x)$ to $(G/K, y)$ in $\Pi_G(X)$.    So $\gamma:  x \to \alpha y$ and $\zeta:  x \to \beta y$ in $X^H$.  And suppose there exists a 2-cell from  $p(\alpha, [\gamma])$ to $p(\beta, [\zeta])$ in $\Pi_{\bar{G}}(\bar{X})$.   This means that we have  $\bar{\sigma}$ in  $\overline{(G/K)^H}$ from $\bar{\alpha}$ to $\bar{\beta}$ such that $\bar{\gamma} * \bar{\sigma} \bar{y}$ is homotopic to $\bar{\sigma}$ via $\bar{\Lambda}$ in $\bar{X}^{\bar{H}}$:  

$$
\xymatrix{
\bar{x} \ar[r]^{\bar{\gamma}} \ar[dr]_{\bar{\zeta}}^{\bar{\Lambda}} &
\bar{\alpha} \bar{y} \ar[d]^{\bar{\sigma} \bar{y}}
\\
& \bar{\beta} \bar{y}
}
$$
 Now $y \in X^K$, so $K \cap N = \{ e\}$ since $N$ acts freely.  So the map $(G/K)^H \to \overline{(G/K)^H}$ is a principal fibration and we can lift $\bar{\sigma}$ to $\sigma$ in $(G/K)^H$ starting at $\alpha$.  
Next, we can lift the homotopy $\bar{\Lambda}$ to $X^H$ starting at $\gamma * \sigma y$ to get a homotopy $\Lambda$:  
$$
\xymatrix{
x \ar[r]^{{\gamma}} \ar[dr]_{{\omega}}^{\Lambda} &
{\alpha}{y} \ar[d]^{{\sigma} {y}}
\\
& z
}
$$  Now we have a new end path $\Lambda(s, 1) = \omega$ such that $\bar{\omega} = \bar{\zeta}$ in $\bar{X}^{\bar{H}}$.  
 By Lemma \ref{L:lift}, there is a continuous $\eta:    I \to N$ such that $\zeta = \eta \omega$, and $\eta(t) \in N(H)$.    We will use this $\eta$ to adjust $\Lambda$:  define $\Lambda' (s, t) = \eta(t) \Lambda(s,t)$.    

$$
\xymatrix{
x \ar[r]^{{\gamma}} \ar[dr]_{\zeta = \eta{\omega}}^{\Lambda'} &
{\alpha}{y} \ar[d]^{{\eta \sigma} {y}}
\\
&  \beta y
}
$$  
We claim that  $\eta \sigma$ together with the homotopy $\Lambda'$ gives a 2-cell in $\Pi_G(X)$.  So we need that $\eta \sigma \in (G/K)^H$ and $ \Lambda'$ is in $X^H$.   This follows from the fact that  $\sigma$ satisfies $\sigma^{-1} H \sigma \subseteq K$ and $\eta(t)  \in N(H)$, so $(\eta \sigma)^{-1} H (\eta \sigma) = \sigma^{-1} H \sigma \subseteq K$; similarly  $\Lambda \in X^H$, so $\Lambda' = \eta \Lambda \in X^{\eta^{-1} H \eta} = X^H$.

{\bf Faithful on 2-cells:}   Suppose we have $\sigma_1$ and $\sigma_2$ representing 2-cells from $(\alpha, [\gamma])$ to $(\beta, [\zeta])$ in $\Pi_G(X)$.    So   there are homotopies $\Lambda_1:  \zeta \simeq \gamma * \sigma_1 y$ and $\Lambda_2:  \zeta \simeq \gamma * \sigma_2y$ in $X^H$.  And suppose that $[\bar{\sigma}_1] = [\bar{\sigma}_2]$, meaning that there is a   homotopy $\bar{\Omega}$ between $\bar{\sigma}_1$ and $\bar{\sigma}_2$ in $\overline{(G/K)^H}$.   Let $Y= (G/K)$;  then since $K \cap N = \{ e\}$, $Y$ is a free $N$-space, and so the projection $p$ is a  principal fibration and  we can lift $\bar{\Omega} $ in $\bar{Y}^{\bar{H}}$ to  $\Omega$ in $Y^H$, starting at $\sigma_1$.  Thus we have a homotopy from  from $\sigma_1$ to $\omega$, where $\bar{\sigma}_2 = \bar{\omega}$.  Applying Lemma \ref{L:lift} gives $\eta(t):  I \to   N$  such that $\eta  \omega = \sigma_2$ and $\eta \in N(H)$.  Again, we define a new homotopy $\Omega'(s,t)  = \eta(t) \Omega(s,t)$, and note that  since $\Omega \in Y^H$, so is   $\eta  \Omega$.  Thus $\Omega'$ gives a homotopy from  $\sigma_1 \simeq \sigma_2$ in $(G/K)^H$, showing that $[\sigma_1] = [\sigma_2]$ in $\Pi_G(X)$.    
\end{proof}

Now we consider the other case, where we have an $L$ space $X$  and take the inclusion of $X$ to $G \times_L X$.    For this, we need the following lemmas. 
\begin{lemma}\label{L:reparam1}  Given a path $\gamma$ in $(G \times_L X)^H$ defined by $\gamma(t) = (g_t, x_t)$ with $g_0 = e$, there exists a continuous path $\ell_t \in L$ with $\ell_0 = e$, such that \begin{itemize}
    \item $(g_t\ell_t)^{-1} H g_t\ell_t = H$ and
    \item $(g_t\ell_t)^{-1} \gamma $ lies in the image of $X^H$ in $(G \times_LX)^H$.  
   \end{itemize}
\end{lemma}
\begin{proof}
 Suppose we have a path  $\gamma$ in $(G \times_L X)^H$ defined by $\gamma(t) = (g_t, x_t)$ with $g_0 = e$.  By the equivariant slice theorem, at any point $x \in X$, there exists a neighbourhood of the form $L \times_{I_x} Y$ for $I_x$ the isotropy group of $x$, and $Y$ an $I_x$ space.  By compactness, we can cover the path $x_t$ with a finite number of such neighbourhoods $N_0 N_1, N_2, \dots, N_k$ around $x_0, x_{s_1}, x_{s_2}, \dots, x_{s_k}$.     Now for each neighbourhood, $\gamma$ takes the form $(g_t, (n_t, y_t))$ where $y_t \in Y_j$ and $Y_j$ is an  $I_{x_{s_i}}$-space.  Our strategy will be to choose, for each neighbourhood, a 'transition' element $\hat{\ell}_j$ and then define $\ell_t = n_t \hat{\ell}_j$.   We will also choose  transition points $x_{t_1}, x_{t_2}, x_{t_3}, \dots, x_{t_n}$ where each $x_{t_i}  \in N_{i-1} \cap N_i$; these are the points where we will switch from using one coordinate system to another.   
 
 Start with the first neighbourhood, so $\gamma = (g_t, x_t)$ and $x_t = (n_t, y_t)$ for a continuous choice of $n_t \in L$ and $ y_t \in Y_0$.   Now by assumption,  $(g_t, (n_t, y_t)) \in (G \times_L X)^H$, so $x_t = (n_t, y_t) \in X^{g_t^{-1} H g_t}$ and therefore  $y_t \in Y^{\ell_t^{-1} g_t^{-1} H g_t \ell_t}$;  note that $g_0 =e$ and $n_0 = e$.  Now ${\ell_t^{-1} g_t^{-1} H g_t \ell_t}$ forms a continuous family of conjugate subgroups of $I_{x_0}$, starting from $H$ at $t=0$;  and since $I_{x_0}$ is discrete, this family must be constant. So for the first neighbourhood, we take the transition element $\hat{\ell}_0 = e$, and define $\ell_t = n_t$ for $0 \leq t \leq t_1$,  so that   ${\ell_t^{-1} g_t^{-1} H g_t \ell_t} = H$ and $(g_t \ell_t)^{-1} \gamma = (e, (e, y_t)) $ lies in the image of $X^H$.      
 
 Now we will inductively assume   that the neighbourhood around $x_{s_{j-1}}$ gives coordinates $(m_t, y_t)$ for $y_t \in Y$ for $Y$ an $I_{x_{s_{j-1}}}$ space, and that $\ell_t = m_t \hat{\ell}_{j-1}$ has been defined for $0 \leq t_{j}$.  We want to extend to $N_j$,  given by coordinates $(n_t, z_t)$ where $z_t \in Z$ for $Z$ an $I_{x_{s_{j}}}$ space.    We use the fact that at our chosen transition point in the intersection, $x_{t_j}$ can be written both as $(m_{t_j}, y_{t_j})$ and $(n_{t_j}, z_{t_j})$, and define our new transition element $\hat{\ell}_{j} = n_{t_j}^{-1} m_{t_j} \hat{\ell}_{j-1}$.    So for $t_j \leq t \leq t_{j+1}$ we will define $\ell_t = n_t\hat{\ell}_j$; this will be continuous since at $t = t_j$, we have $\ell_{t_j} = m_{t_{j-1}}\hat{\ell}_{j-1} = n_{t_j} n_{t_j}^{-1} m_{t_{j-1}}\hat{\ell}_{j-1} = n_{t_j} \hat{\ell}_{j}$.  
 
 Now we know by assumption,  $(g_{t_j} \ell_{t_j})^{-1} H g_{t_j} \ell_{t_j} = H$, and so $$\hat{\ell}^{-1}_{j-1} m^{-1}_{t_j} g^{-1}_{t_j} H g_{t_j} m_{t_j} \hat{\ell}_{j-1} = H$$  Therefore $$m^{-1}_{t_j} g^{-1}_{t_j} H g_{t_j} m_{t_j} = \hat{\ell}_{j-1} H \hat{\ell}^{-1}_{j-1}$$ and so \begin{align*} n^{-1}_{t_j}  g^{-1}_{t_j} H g_{t_j} n_{t_j} & =n^{-1}_{t_j} m_{t_j} m^{-1}_{t_j} g^{-1}_{t_j} H g_{t_j} m_{t_j} m^{-1}_{t_j} n_{t_j} \\ & = n^{-1}_{t_j} m_{t_j} \hat{\ell}_{j-1} H \hat{\ell}_{j-1} m^{-1}_{t_j} n_{t_j} \\ & = \hat{\ell}_j H \hat{\ell}_j^{-1}\end{align*}  So in $N_j$,  we have a family of subgroups $n^{-1}_{t}  g^{-1}_{t} H g_{t} n_{t} $ of $I_{x_{s_j}}$ which is equal to  $ \hat{\ell}_j H \hat{\ell}_j^{-1}$ at $t_j$,  and so by discreteness of $I_{x_s}$ they must be constant.  Thus we have that $$\hat{\ell}_j^{-1}  n^{-1}_{t}  g^{-1}_{t} H g_{t} n_{t}  \hat{\ell}_j = H$$ and so  $ (g_{t} \ell_{t})^{-1} H g_{t} \ell_{t} = H $ for $t_j \leq t \leq t_{j+1}$, and  $$(g_t\ell_t)^{-1} \gamma = \hat{\ell}_j^{-1} n_t^{-1} g_t^{-1} (g_t, (n_t, z_t)) = (e, (\hat{\ell}_j, z_t)) $$ which lies in the image of  $ X^H.$

\end{proof}

\begin{lemma}\label{L:reparam2}  Given a homotopy $\Lambda$ in $(G \times_L X)^H$ defined by $\Lambda(t) = (g_{s,t}, x_{s,t})$ with $g_{0, 0} = e$, there exists a continuous family  $\ell_{s,t} \in L$ such that \begin{itemize}
    \item $(g_{s,t}\ell_{s,t})^{-1} H g_{s,t}\ell_{s,t} = H$ and 
    \item $(g_{s,t}\ell_{s,t})^{-1} \Lambda $ lies in the image of $X^H$.  

\end{itemize}
\end{lemma}
\begin{proof}
We use a similar strategy as in the proof of Lemma \ref{L:reparam1}, working with local neighbourhoods of the form $L \times_I Y$ and producing adjustment elements that allow us to piece together the local coordinates into a continuous whole.    So we begin by covering the image of $x_{x,t}$ by a set of neighbourhoods of the form $L \times_{I_x} Y$ where $I_x$ is the isotropy of a point $x$, and $Y$ is a space with $I_x$ action.  By compactness, we will assume that the number of these neighbourhoods is finite, and that they cover a grid pattern where the horizontal and vertical sides of the images of the grid are in overlapping neighbourhoods:

 \begin{tikzpicture}
\draw[thick] [-] (0,0) -- (3,0);
\draw[thick] [-] (0,0) -- (0,2);
\draw[thick] [-] (1,0) -- (1,2);
\draw[thick] [-] (2,0) -- (2,2);
\draw[thick] [-] (3,0) -- (3,2);
\draw[thick] [-] (0,1) -- (3,1);
\draw[thick] [-] (0,2) -- (3,2);
\draw (.25,.25) circle [radius=1.2];
\draw (1.5,.25) circle [radius=1.1];
\draw (2.75,.25) circle [radius=1.2];
\draw (.25,1.75) circle [radius=1.2];
\draw (1.5,1.75) circle [radius=1.2];
\draw (2.75,1.75) circle [radius=1.2];

\end{tikzpicture}

In the first neighbourhood, centered at $x_{0, 0}$, we have coordinates $(n_{s,t}, y_{s,t})$ and can define  $\ell_{s,t} = n_{s,t}$; note that $\ell_{0,0} = e$.     To extend to new neighbourhoods, we work across the top row, then down the left row, and then through the middle of the grid, row by row.  So we need to be able to extend in the following scenarios:  when the $\ell_{s,t}$ have been defined for the neighbourhood to the left;  when the  $\ell_{s,t}$ have been defined for the neighbourhood above;  and when  $\ell_{s,t}$ have been defined for both  the neighbourhood to the left and the neighbourhood above.  In each case, we will use the grid edges in the overlap of adjacent neighbourhoods to give us our new adjustment elements.  These overlaps occur along lines instead of single points, so our adjustment elements will not be constant on each neighbourhood, but rather an extension of the edge adjustment lines.     All of our extensions will produce continuous adjustment factors, hence continuous choices of $\ell_{s,t}$, and the discreteness of the isotropy group $I_x$ on each neighbourhood will ensure that we retain the properties we need.  

To extend across the top row of neighbourhoods, we assume that we have defined our $\hat{\ell}_{s,t}$ on the neighbourhood to the left, and look along the transition line at $s = S$, along the segment $x_{S, t}$ for $0 \leq t \leq t_1$.  Then we have two neighbourhood parametrizations, $(m_{s, t}, y_{s, t})$ for the left neighbourhood and $(n_{s, t}, z_{s, t})$ for the right where we are looking to extend;  and the vertical  transition line at $s = S$.   We define the new transition element using the values along this transition line, constant horizontally, by   $\hat{\ell}_{s,t} = n^{-1}_{S, t} m_{S, t} \hat{\ell}_{S, t}$: 

  \begin{tikzpicture}
\draw[thick] [-] (0,0) -- (2,0);
\draw[thick] [-] (0,0) -- (0,1);
\draw[thick] [-] (1,0) -- (1,1);
\draw[thick] [-] (2,0) -- (2,1);
\draw[thick] [-] (0,1) -- (2,1);
\draw [-] (1,.2) -- (2,.2);
\draw [-] (1,.4) -- (2,.4);
\draw [-] (1,.6) -- (2,.6);
\draw [-] (1,.8) -- (2,.8);
\node at (0,-.6){$(m_{s,t}, y_{s,t})$};
\node at (2,-.6){$(n_{s,t}, z_{s,t})$};
\node at (1,1.6){$s=S$};
\draw (.5,.5) circle [radius=.9];
\draw (1.5,.5) circle [radius=.9];
\end{tikzpicture}

Similarly, to extend down the left row, we use neighbourhood parametrizations $(m_{s,t}, y_{s,t})$ for the upper and $(n_{s,t}, z_{s,t})$ neighborhood, and look at the horizontal  transition line at $t = T$.  We define the new transition elements using values along this line, constant vertically, by  $\hat{\ell}_{s,t} = n^{-1}_{s, T} m_{s, T} \hat{\ell}_{s, T}$: 

  \begin{tikzpicture}
\draw[thick] [-] (0,0) -- (1,0);
\draw[thick] [-] (0,0) -- (0,2);
\draw[thick] [-] (1,0) -- (1,2);
\draw[thick] [-] (0,2) -- (1,2);
\draw[thick] [-] (0,1) -- (1,1);
\draw [-] (.2,0) -- (.2,1);
\draw [-] (.4,0) -- (.4,1);
\draw [-] (.6,0) -- (.6,1);
\draw [-] (.8,0) -- (.8,1);
\node[align=right] at (2.3,1.5){$(m_{s,t}, y_{s,t})$};
\node[align=right] at (2.3,.5){$(n_{s,t}, z_{s,t})$};
\node[align=left] at (-.8,1){$t=T$};
\draw (.5,.5) circle [radius=.9];
\draw (.5,1.5) circle [radius=.9];
\end{tikzpicture}

Lastly, consider how to extend to a neighbourhood where $\hat{\ell}_{s,t}$ has been defined for a neighbourhood above and to the left of it. Suppose that the neighbourhood above our new one has coordinates $(m_{s,t}, y_{s,t})$, and the one to the left has coordinates $(m'_{s,t}, y'_{s,t})$; and the new neighbourhood has coordinates $(n_{s,t}, z_{s,t})$.   Then we have two transition segments.  Along the horizontal segment at $t = T$, we know that  $x_{s, T}$ can be written as both $(m_{s, T}, y_{s, T})$ and $(n_{s, T}, z_{s, T})$; and along the vertical segment $s=S$,  we know that  $x_{S, t}$ can be written as  $(m'_{S, t}, y'_{S, t})$ and $(n_{S, t}, z_{S, t})$.      Then we can extend  diagonally across the new neighbourhood using transition element $$\hat{\ell}_{s, t}  = \begin{cases} n^{-1}_{s-t, T}m_{s-t, T} \hat{\ell}_{s-t, T}  & { \textup{if}  }\,\,  s-S \leq t-T \\  n^{-1}_{S, s-t}m'_{S, s-t} \hat{\ell}_{S, s-t} &  { \textup{else}  }\\  \end{cases} $$ 
  \begin{tikzpicture}
\draw[thick] [-] (0,0) -- (2,0);
\draw[thick] [-] (0,0) -- (0,1);
\draw[thick] [-] (1,0) -- (1,1);
\draw[thick] [-] (0,1) -- (2,1);
\draw[thick] [-] (0,1) -- (1,1);
\draw[thick] [-] (1,1) -- (1,2);
\draw[thick] [-] (2,0) -- (2,2);
\draw[thick] [-] (2,1) -- (2,2);
\draw[thick] [-] (1,2) -- (2,2);
\draw [-] (1,1) -- (2,0);
\draw [-] (1.2,1) -- (2,.2);
\draw [-] (1.4,1) -- (2,.4);
\draw [-] (1.6,1) -- (2,.6);
\draw [-] (1.8,1) -- (2,.8);
\draw [-] (1,.8) -- (1.8,0);
\draw [-] (1,.6) -- (1.6,0);
\draw [-] (1,.4) -- (1.4,0);
\draw [-] (1,.2) -- (1.2,0);
\node[align=right] at (3.3,1.7){$(m_{s,t}, y_{s,t})$};
\node[align=right] at (0,-.8){$(m_{s,t}', y_{s,t}')$};
\node[align=right] at (3.3,.2){$(n_{s,t}, z_{s,t})$};
\node[align=left] at (-.8,1){$t=T$};
\node at (1,2.5){$s=S$};
\draw (.5,.5) circle [radius=.9];
\draw (1.5,.5) circle [radius=.9];
\draw (1.5,1.5) circle [radius=.9];

\end{tikzpicture}

 In this way we can define our transition elements on the entire region.

\end{proof}
\begin{prop}  \label{P:incl} If $L$ is a compact Lie group which acts smoothly on $X$ with finite isotropy groups, and  $L \subseteq G$ for some compact Lie group $G$,  then there is an inclusion map $L \into G$, and an equivariant inclusion  $\varphi:  X \to G \times_L X$ by $\varphi(x) = [e, x]$.  Then $\varphi:  \Pi_L(X) \to \Pi_{G}(G \times_L X)$ is an equivalence of $2$-categories.  
\end{prop}

\begin{proof}
 {\bf Essentially Surjective on Objects:}  Given $(G/H, [g, y])$ in $\Pi_G(G \times_LX)$, we know that  $[g, y] \in (G \times_LX)^H$, and so $[e, y] \in (G\times_L X)^{g^{-1}Hg}$.  So $y \in X^{g^{-1}Hg}$.  Consider $(L/g^{-1}Hg, y) \in \Pi_L (X)$.  Then $\varphi(L/g^{-1}Hg, y) = (G/g^{-1}Hg, [e, y])$ and $(g, c)$ represents an arrow from $(G/H, [g, y])$ to  $(G/g^{-1}Hg, [e, y])$,  where $c$ is the constant path at $[g, y]$.  
 
 {\bf Essentially Surjective on Arrows:}  Suppose we have  $(L/H, x)$ and $(L/K, y)$ in $\Pi_L(X)$, and an arrow $(\alpha, [\gamma]):  (G/H, [e,x]) \to (G/K, [e, y])$ in $\Pi_G(G \times_L X)$.  So  we have a path  $\gamma:  [e, x] \to [\alpha, y]$ in $(G \times_L X)^H$, and $\gamma$ takes the form  $\gamma(t) = [g_t, x_t]$.    By Lemma \ref{L:reparam1}, we  know that there exists a continuous family   $\ell_t$ such that  $(g_t\ell_t)^{-1} H g_t\ell_t = H$ and $(g_t\ell_t)^{-1} \gamma \in (e \times_L X^H)$.  Define $\zeta =(g_t\ell_t)^{-1} \gamma $.  Then $\zeta$ is a path in $X^H$ which starts at $x$ and ends at $\ell_1^{-1} y$, and so  $(e, (\ell_1^{-1}, \zeta))$ is an arrow in $\Pi_L(X)$ from $(L/H, x)$ to $(L/K, y)$.  Now define a 2-cell  $\sigma:  (g_t\ell_t)^{-1}$ from $(\alpha, \gamma)$ to $\varphi(\ell_1^{-1}, \zeta)$:  to show this is a 2-cell in $\Pi_G(G \times_L X)$,   we need a homotopy $\gamma * \sigma y  \simeq \zeta$.  We  define $\Lambda(s, t) = \sigma(s) \gamma(t)$: 
  $$
\xymatrix{
x \ar[r]^{{\gamma}} \ar[dr]_{\zeta = \sigma \gamma} &
{\alpha}{y} \ar[d]^{{\sigma} {y}}
\\
&  \ell^{-1}  y
}
$$  
Thus every arrow between $\varphi(L/H, x)$ and $\varphi(L/K, y)$ has an invertible 2-cell to an arrow in the image of $\Pi_L(X)$.  

{\bf Full on 2-cells}  Suppose we have two arrows $(\alpha, [\gamma])$ and $(\beta, [\zeta])$ from $(L/H, x)$ to $(L/K, y)$ in $\Pi_L(X)$, and that when we take their images under $\varphi$, there is a 2-cell connecting them in $\Pi_G(G \times_L X)$.  Specifically, $\gamma$ is a path from $x$ to $\alpha y$ in $X^H$, and $\zeta$ is a path from $x$ to $\beta y$ in $X^H$;  and there is a  $\sigma$ from $\alpha$ to $\beta$ in $(G/K)^H$, and a homotopy  $\Lambda$ from $ [e, \gamma] *[\sigma, y]$ to $[e, \zeta]$ in $(G \times_L X)^H$. Now apply Lemma \ref{L:reparam2}  to   $\Lambda(s,t) = [g_{s,t}, x_{s,t}]$   to get a continuous  $\ell_{s,t} \in L$ such that  $(g_t\ell_t)^{-1} \gamma \in (e \times_L X^H)$  and  $(g_t\ell_t)^{-1} H g_t\ell_t = H$.    

Define $\omega(t) = (g_{1, t} \ell_{1, t})^{-1} \sigma(t)$ and $\Omega(s,t) = (g_{s,t}\ell_{s,t})^{-1} \Lambda(s,t)$.  Then  $\omega$ is a map in $(L/K)^H$ from $(g_{1,0}\ell_{0,1})^{-1}\alpha$ to $(g_{1,1}\ell_{1,1})^{-1}\beta$, and $\Omega$ gives a homotopy in $X^H$ from  $(g_{s,0}\ell_{s,0})^{-1}\gamma * \omega y$ to $(g_{s,1}\ell_{s,1})^{-1} \zeta$. 

 $$
\xymatrix{
[e,x] \ar[rrr]^{(g_{s,0}\ell_{s,0})^{-1}{\gamma}} \ar[ddrrr]_{(g_{s,1}\ell_{s,1})^{-1} \zeta} & &  &
{(g_{1,0}\ell_{1,0})^{-1} \alpha}{y} \ar[dd]^{{\omega} {y}}
\\ \\
&&&  (g_{1,0}\ell_{1,0})^{-1}\beta  y
}
$$  We want to use this to create a 2-cell in $\Pi_L(X)$ from $(\alpha, [\gamma])$ to $(\beta, [\zeta])$.   Now we know that $\gamma \in X^H$, and  $[e, \gamma(s)] = \Lambda_{s,0} = (g_{s,0}, x_{s,0})$.  So we must have $g_{s,0} \in L$.  Thus we can define a 2-cell in $\Pi_L(X)$ from  $(\alpha, [\gamma])$ to $((g_{1,0}\ell_{1,0})^{-1} \alpha, [(g_{s,0}\ell_{s,0})^{-1}{\gamma}])$ using $g_{s,0} \ell_{s,0}$.  Similarly, on the other end, we have a 2-cell from $(g_{1,1}\ell_{1,1})^{-1} \beta, [(g_{s,0}\ell_{s,0})^{-1}{\zeta}])$ to   $(\beta, [\zeta])$ using $g_{s,1} \ell_{s,1}$, again using the fact that $ \Lambda_{s,1} = [e, \zeta]$ and hence $g_{s,1} \in L$.  

{\bf Faithful on 2-cells}  Suppose we have two 2-cells $\sigma, \omega$ in $\Pi_L(X)$ from $(\alpha, [\gamma])$ to $(\beta, [\zeta])$.    And when we embed in $G \times_L X$, they are homotopic:   we have $\Gamma:  \sigma \to \omega$ in $(G/K)^H$, which gives $\Gamma y:  \sigma y \to \omega y$ in $(G \times_L X)^H$.   Note that when we write $\Gamma y$ in the form $(g_{s,t}, x_{s,t})$ we can take $g_{s,t} = \Gamma(s,t)$ and $x_{s,t} = y$.     Apply Lemma \ref{L:reparam2} to $\Gamma y$ written in this form to get $\ell_{s,t}$;  then define $\hat{\Gamma} = (g_{s,t} \ell_{s,t})^{-1} \Gamma = \ell_{s,t}^{-1}$.  Because $\Gamma \in (G/K)^H$, we know that $g_{s,t} H g_{s,t}^{-1}  \leq K$, and so  $\ell_{s,t}^{-1} g_{s,t}^{-1} H g_{s,t} \ell_{s,t} \leq \ell_{s,t}^{-1} K \ell_{s,t}$ and so $\ell_{s,t}^{-1}$ gives a map from $L/H$ to $L/K$ as required;  giving a 2-cell in $\Pi_L(X)$ from $g_{s,0}^{-1} \sigma$ to $(g_{s,1}\ell_{s,1})^{-1} \omega$.  Then since $\sigma \in L$, we know that $g_{s,0} \in L$, allowing us to create a homotopy from $ g_{s,0}^{-1} \sigma$ to $\sigma$; and similarly from $(g_{s,1}\ell_{s,1})^{-1} \omega $ to $\omega$ in $(L/K)^H$.    Concatenating these homotopies gives the required homotopy from $\sigma$ to $\omega$, showing that these 2-cells were equivalent in $\Pi_L(X)$. 
\end{proof}
\begin{eg}
Suppose that $G = \mathbb{Z}/2$ acts on the circle $X = S^1$, where $\tau$ acts by rotation by $\pi$.    Then $G$ acts freely on $X$, with quotient space $\bar{X} \simeq S^1$,  and so the projection $G \ltimes X \to e \ltimes \bar{X}$ is an equivariant essential equivalence.  

Now $\bar{G} = \{e\}$ and so $\Pi_{\bar{G}}{\bar{X}}$ is just the usual fundamental groupoid of $\bar{X}$, with arrows indexed by the winding number of their path in $\mathbb{Z}$  between any two points.  Since $G$ is discrete, $\Pi_{\bar{G}}{\bar{X}}  = \Pi^d_{\bar{G}}{\bar{X}}$.    

On the other hand, $\Pi_G(X)$ consists of points $(G/e, x)$ and arrows $(e, [\gamma])$ where $\gamma$ is a path from $x$ to $y$, again classified by winding number $n$;  and arrows $(\tau, [\zeta])$ where $\zeta$ is a path from $x$ to $\tau y$, classified by winding number $m$.   Then the projection map takes $(e, [\gamma])$ to $(\bar{e}, [\bar{\gamma}])$, which will have winding number $2n$;  and takes $(\tau, [\zeta])$ to $(\bar{e}, [\bar{\zeta}])$ with winding number $2m+1$.  Again,  $\Pi_G(X) = \Pi^d_G(X)$.  So we have the equivalence of catgories $\Pi^d_G(X) \longrightarrow  \Pi^d_{\bar{G}}{\bar{X}}$.
\end{eg}

\begin{eg}
 Consider $G = S^1$ acting on the torus $X = T^2$ as in Example \ref{E:t2}.  The circle acts freely, so again the projection  $G \ltimes X \to e \ltimes \bar{X}$ is an equivariant essential equivalence.  Here $\bar{X} \simeq S^1$ and again $\Pi_{\bar{G}}{\bar{X}} =\Pi^d_{\bar{G}}{\bar{X}}  $ is the groupoid with arrows indexed by  $\mathbb{Z}$  between any objects.  
 
 The category $\Pi_G(X)$ was described in Example \ref{E:t2}, where it was shown that every arrow had a 2-cell to an arrow of the form $(e, [\gamma])$ where $\gamma$ was constant in the second coordinate.  Therefore $\Pi^d_G(X)$ isomorphic to $\Pi^d_{\bar{G}}{\bar{X}}$.  
\end{eg}

\begin{eg}
 Consider $L = \mathbb{Z}/2$ and $X = S^1$ from Example \ref{E:s1}.  We can consider $L$ as a subgroup of $G = S^1$, and form the space $G \times _{L} X = (G \times X)/ \sim$, where $(\lambda, x) \sim (\tau \lambda, \tau x) $ for $\tau = e^{i\pi}$.  Since $\tau$ acts on $X$ via a reflection, this is topologically a Klein bottle. 
 
 The objects of $\Pi_G(G \times_L X)$ are of the form $(G/e, [\alpha, \beta])$ for $[\alpha, \beta] \in S^1 \times_L S^1$, and $G/G, [\alpha, z])$ for $z  \in X^{{\mathbb Z} /2}$, so $z \in \{ E, W\}$. 
 
 Between points of the first type, we have arrows of the form $(\lambda, [\gamma])$ for $\gamma$ a path from $x$ to $\lambda y$ in the Klein bottle; so we have a copy of $S^1 \times ({\mathbb Z} \oplus {\mathbb Z})/<abab^{-1}>$.   Similarly the points $(G/G, [\alpha, z])$ have arrows $(\lambda, [\zeta])$ between them, where $\zeta$ is a path from $(\alpha, z)$ to $(\alpha', z)$ in $G \times_L X)^{{\mathbb Z}/2} \simeq S^1$; and arrows $(\lambda, [\gamma])$ from $(G/e, [\alpha, \beta])$ to $(G/G, [\alpha, z])$.  
 
 There are 2-cells identifying many of these arrows.  Between $(G/e, [\alpha_1, \beta_1])$ and $(G/e, [\alpha_2, \beta_2])$,  there are two-cells between arrows $(\lambda, [\gamma])$ and $(e, [\gamma * \Lambda^{-1}  (\alpha_2, \beta_2)$ by defining $\Lambda$ to be a  path in $S^1$ from $e$ to $\lambda$. \[\scalebox{.15}{\includegraphics{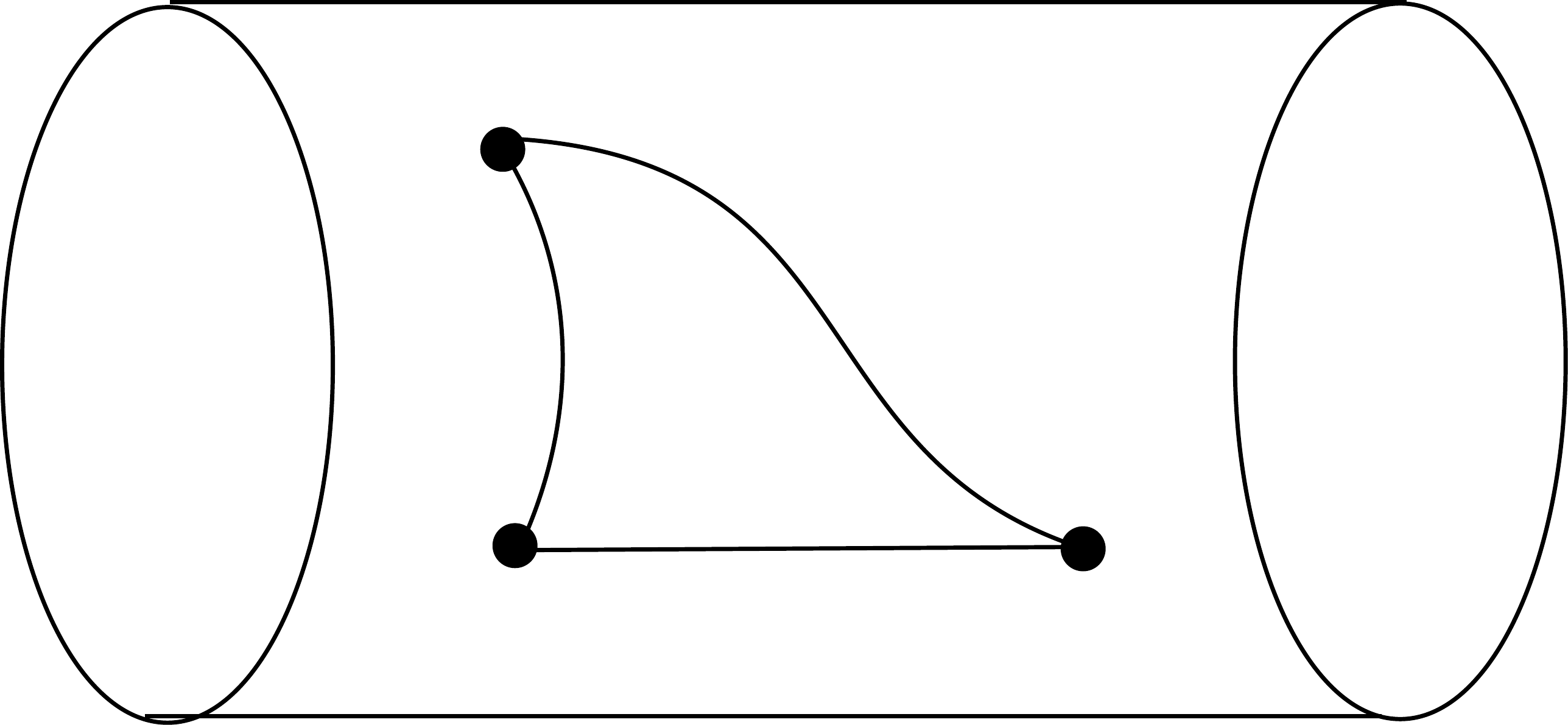}}\]   Now consider the arrow  $(e, [b])$ corresponding to a generating loop along $\beta$.  This has a 2-cell to the map $(\tau, [\hat{b}]) $  which is constant in $\beta$, where $\hat{b}$ is a path from $x$ to $\tau x$. \[\scalebox{.15}{\includegraphics{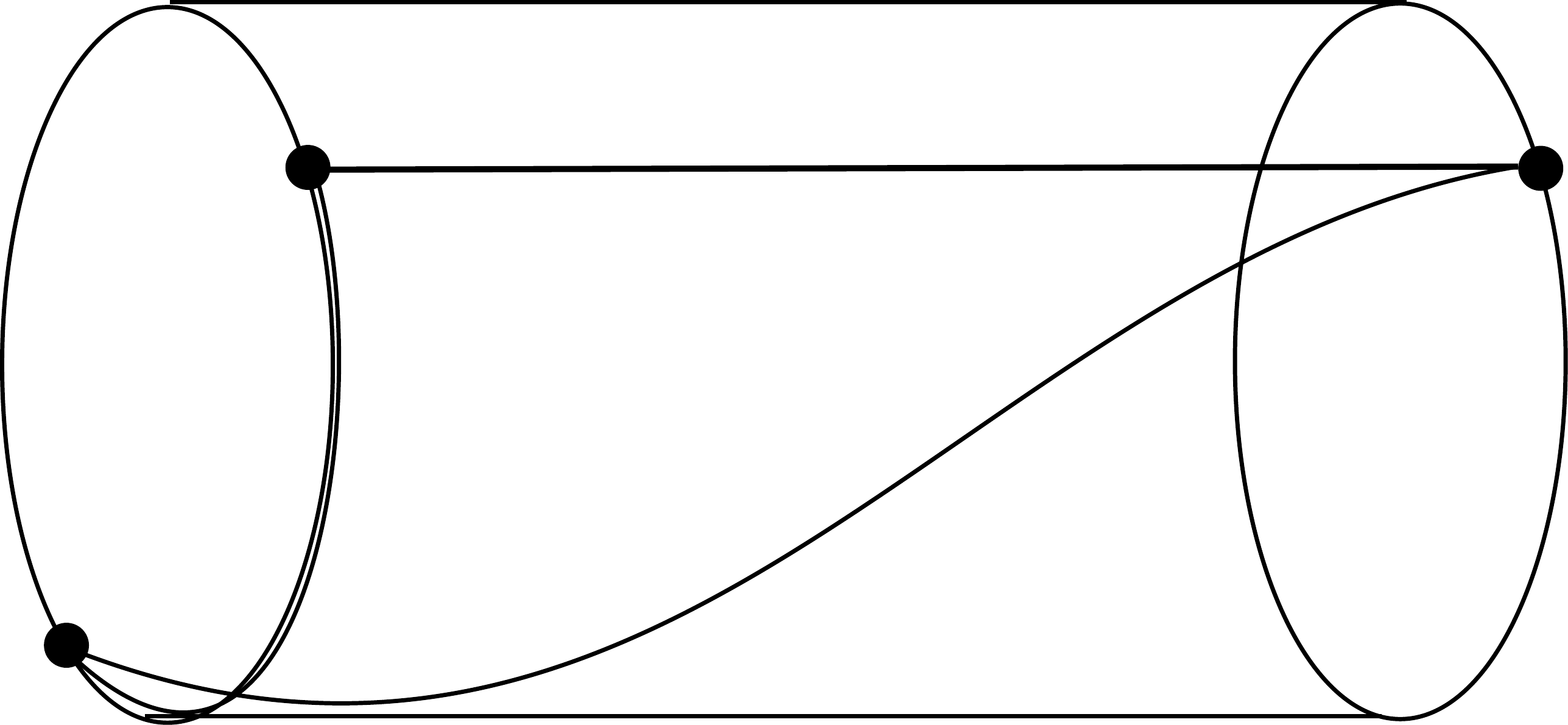}}\]  Then   $(\tau, [\hat{b}])^2 = e $.  Thus in $\Pi^d_G(G \times_L X)$, we have arrows indexed by elements of the group ${\mathbb Z} \oplus {\mathbb Z}/<abab^{-1}, b^2>$, recovering the infinite dihedral maps of Example \ref{E:s1}.    Similarly,  arrows   from $(G/{\mathbb Z}/2, [\alpha, z])$ to $(G/{\mathbb Z}/2, [\alpha', z])$ are  identified  in $\Pi^d_G(G \times_L X)$ by taking a path from $\alpha$ to $\alpha'$ in $S^1$ and using it to produce a 2-cell, and we get our equivalence $\Pi^d_L(X) \to \Pi^d_G(G \times_L X)$.   
\end{eg}

\end{document}